%% file: Main.tex
\documentclass[reqno]{amsart}
\usepackage[foot]{amsaddr}

\usepackage{mathabx, amscd}
\usepackage[utf8]{inputenc}
\usepackage{amsmath}
\usepackage{amsfonts}
\usepackage{amsthm}
\usepackage{amssymb}
\usepackage{xy}\xyoption{all}

\usepackage[normalem]{ulem}
\usepackage{enumerate}
\usepackage{tikz}
\usepackage{xcolor}
\usepackage[hidelinks]{hyperref}
\usepackage{mathrsfs}
\usepackage{todonotes}
\hypersetup{colorlinks,breaklinks,
             urlcolor=blue,
             linkcolor=blue}

\theoremstyle{plain}
\newtheorem{thm}{Theorem}[section]
\newtheorem{lem}[thm]{Lemma}
\newtheorem{pro}[thm]{Proposition}
\newtheorem{cor}[thm]{Corollary}
\theoremstyle{definition}
\newtheorem{definition}[thm]{Definition} 
\newtheorem{rem}[thm]{Remark}

\newtheorem{exa}[thm]{Example}

\newcommand{\lo}{\prec_l}

\newcommand{\pq}{\prec_q}

\newcommand{\chain}[2]{ \{ {#1}_0\subset  {#1}_1\subset \dots \subset {#1}_{#2} \} }

\newcommand{\citep}{\cite}
\newcommand{\F}{\mathbb{F}}

\newcommand{\Fqn}{\Fq^n}
\newcommand{\Fq}{\F_q}

\newcommand{\<}{\left<}
\renewcommand{\>}{\right>}

\newcommand{\bbZ}{\mathbb{Z}}
\newcommand{\NN}{\mathbb{N}}

\newcommand{\R}{\mathcal{R}}

\newcommand{\B}{\mathcal{B}}

\newcommand{\M}{\mathbf{M}}
\newcommand{\T}{\tau}
\newcommand{\TT}{\mathcal{T}}

\newcommand{\I}{\mathcal{I}}
\newcommand{\0}{\mathbf{0}}

\newcommand{\ZZ}{\mathbb{Z}}

\newcommand{\SE}{\Sigma(E)}
\newcommand{\GrE}{\mathbb{G}_r(E)}

\newcommand{\UU}{\mathfrak{U}}
\newcommand{\VV}{\mathfrak{V}}
\newcommand{\WW}{\mathfrak{W}}
\newcommand{\tovo}[1]{{\color{red}Tovo:  #1}}

\newcommand{\Pos}{\mathcal{P}\kern-1pt os} 
\newcommand{\Sset}{s\mathrm{Set}} 
\newcommand{\Top}{\mathcal{T}\kern-1pt op}
\newcommand{\Ex}{\mathrm{Ex}} 
\newcommand{\Cat}{\mathcal{C}at}
\newcommand{\Deltao}{\mathring{\Delta}}
\newcommand{\redH}{\tilde{\mathrm{H}}}
\newcommand{\calA}{\mathcal{A}}
\newcommand{\calH}{\mathcal{H}}
\newcommand{\bbF}{\mathbb{F}}

\DeclareMathOperator{\id}{id}

\title{Homotopy type of shellable $q$-complexes and their homology groups}
\author[Ghorpade]{Sudhir R. Ghorpade$^1$}
\address[1]{Department of Mathematics, Indian Institute of Technology Bombay, Powai, Mumbai 400076, India}
\email{srg@math.iitb.ac.in}
\thanks{Sudhir Ghorpade is partially supported by the grant DST/INT/RUS/RSF/P-41/2021 from the Department of Science \& Technology, Govt. of India}
\author[Pratihar]{Rakhi Pratihar$^2$}
\email{rakhi.pratihar@inria.fr}
\address[2]{Inria Centre de Saclay, Campus Polytechnique, 91120 Palaiseau, France}
\thanks{During the course of this work, Rakhi Pratihar was supported by Grant 280731 from the Research Council of Norway}
\author[Randrianarisoa]{Tovohery H. Randrianarisoa$^3$}
\email{tovo@aims.ac.za}
\address[3]{Department of Mathematics and Statistics and Mathematical Statistics, Umea University, Umea, Sweden}
\author[Verdure]{Hugues Verdure$^4$}
\email{hugues.verdure@uit.no}
\author[Wilson]{Glen Wilson$^4$}
\email{glen.wilson@uit.no}
\address[4]{Department of Mathematics and Statistics, UiT - The Arctic
University of Norway, Tromsø, Norway }
\begin{document}
\maketitle

\input{abstract} 

\input{Introduction}

\input{Combinatorial}

\input{Topology}

\input{Homology_degree}

\input{Conclusion_and_open_question}

\bibliographystyle{plain}
\bibliography{bibliography}

	

\end{document}

%% file: abstract.tex
\begin{abstract}

    The theory of shellable simplicial complexes brings together combinatorics, algebra, and topology in a remarkable way. Initially introduced by Alder for $q$-simplicial complexes, recent work of Ghorpade, Pratihar, and Randrianarisoa extends the study of shellability to $q$-matroid complexes and determines singular homology groups for a subclass of these $q$-simplicial complexes. In this paper, we determine the homotopy type of shellable $q$-simplicial complexes. Moreover, we establish the shellability of order complexes from \emph{lexicographically} shellable $q$-simplicial complexes, that include the $q$-matroid complexes. This results in a comprehensive determination of the homology groups for any lexicographically shellable $q$-complexes.
    \end{abstract}

%% file: Introduction.tex
\section{Introduction}

The notion of shellability garnered significant attention owing to its implications in combinatorics, commutative algebra, and algebraic topology.  Historically, shellability was first implicitly assumed by Schl{\"a}fli \cite{Sch} for proving the Euler-Poincar{\'e} formula for higher-dimensional convex polytopes. However, the shellability of the boundary of convex polytopes was formally established after many decades by Bruggesser and Mani \cite{BM} in 1971.

Informally, for an abstract simplicial complex of dimension $d$, shellability gives a method to construct the complex by gluing one facet at a time to the previously constructed subcomplex in a nice way such that their intersection is topologically a $(d-1)-$sphere. From the topological point of view, a shellable simplicial complex is weak homotopy equivalent to a wedge sum of spheres, thus the homology groups are well understood \cite[\S2.3]{Hatcher}. 
In commutative algebra, shellable simplicial complexes are important, partly because their ``face rings" or Stanley-Reisner rings (over any field) are Cohen-Macaulay. These rings were introduced independently by Hochster and Stanley and they have nice properties. Since Gr\"obner degenerations of the coordinate rings of many classes of algebraic varieties turn out to be Stanley-Reisner rings of simplicial complexes, shellability provides a tool to prove their Cohen-Macaulayness. Motivated by Stanley's work on R-labeling, Bj{\"o}rner introduced the notions of EL-labelling and CL-labelling that give efficient methods for establishing shellability.  Important known classes of shellable simplicial complexes include the boundary complex of a convex polytope \cite{BM}, the order complex of a bounded, locally upper semimodular poset \cite{folkman}, and
matroid complexes, i.e., complexes formed by the independent subsets of matroids \cite{Porvan}. For comprehensive background and the proofs supporting these claims, we refer to the literature, including the books by Stanley \cite{Stan} and Bruns and Herzog \cite{BH}, the compilation of lecture notes in \cite{GSSV}, and Björner's survey article \cite{Bjorner}.

We are interested in $q$-analogs of some of these results wherein finite sets are replaced by finite dimensional vector spaces over finite fields $\Fq$. Our main objects of study are $q$-matroids and $q$-simplicial complexes, which are $q$-analogs of matroids and simplicial complexes, respectively. The $q$-matroids were studied more generally as complementary modular lattices by Crapo \cite{Crapo} and later rediscovered by Jurrius and Pellikaan \cite{JP18} in connection with rank metric codes. On the other hand, the $q$-simplicial complexes found its mention at least in the work of Rota \cite{Rota71} and thereafter, Alder \cite{Alder} studied this objects with more general notion of $q$-simplicial posets. But the very first non-trivial class of shellable $q$-complexes are obtained only recently in \cite{GPR21}, where it is proved that the $q$-matroid complexes, i.e., the $q$-complexes formed by independent spaces of $q$-matroids, are shellable. Furthermore, in \cite{GPR21}, the use of shellability to determine the reduced singular homology  groups has been extended to some classes of $q$-matroid complexes including the $q$-spheres. The results (e.g., \cite[Theorems 5.11, 6.10]{GPR21}) show that there is only one non-trivial homology group at the top degree similar to the classical case of shellable simplicial complexes. This led us to consider the following questions:

\textbf{Question 1:} Does every $q$-matroid complex have only one non-trivial reduced singular homology group? Determine the reduced singular homology groups for any arbitrary $q$-matroid complex.

\textbf{Question 2:} What is the homotopy type of a $q$-matroid complex, more generally, of an arbitrary shellable $q$-complex?

 We also mention that as explained in \cite{GPR21}, studying the topology of shellable $q$-complexes appears to have potential applications to $q$-matroids and rank metric codes. For additional insights on this matter, we refer to \cite[Remark 6.12]{GPR21} and the open question in section \ref{sec:5} at the end of this paper.


In this paper, we give complete answers to both the questions above. For determining the reduced singular homology of an arbitrary $q$-matroid complex, we still make use of shellability, but not that of the $q$-matroid complexes. Rather, we prove that the order complexes associated to $q$-matroid complexes are shellable as simplicial complexes. Thus following the classical method of determining reduced homology of shellable simplicial complexes, we give a complete combinatorial description of the reduced homology of the associated order complexes by counting certain maximal chains. The weak homotopy equivalence between the topological spaces associated to $q$-matroid complexes and their corresponding simplicial complexes implies that they have the same reduced singular homology groups. We make use of this result on weak homotopy equivalence to prove that the homotopy type of an arbitrary shellable $q$-simplicial complex is that of a wedge sum of spheres.

The significance of our main results are threefold: The shellability of order complexes associated to $q$-matroid complexes shows that the $q$-posets underlying $q$-matroid complexes are Cohen-Macaulay posets \cite{bjornerlexi} (See Remark \ref{CMposets} for more on this). Secondly, the homotopy type reveals the topological structure of \emph{any} shellable $q$-complexes, not necessarily the lexicographically shellable ones. Finally, we completely determine the singular homology of $q$-matroid complexes, of which partial results were obtained in \cite{GPR21}. 

 The paper is organized as follows. The subsequent section is divided into two subsections: in subsection \ref{subsec:2.1}, we provide some preliminaries about $q$-complexes and $q$-matroids and review the findings on the shellability of $q$-matroid complexes from \cite{GPR21}. Subsection \ref{subsec:2.2} focuses on proving the shellability of order complexes associated with $q$-matroid complexes. In section \ref{Sec:homotopy}, we consider various topological spaces that can be associated to a $q$-simplicial complex and prove that their homotopy type is that of wedge sum of spheres (cf. Theorem \ref{thm:main}). Section \ref{sec:4} revisits the results on the singular homology of certain classes of shellable $q$-complexes from \cite{GPR21} and explicitly determines the homology of arbitrary $q$-matroid complexes. The procedure we outline in section \ref{sec:4} is applicable to any shellable $q$-complexes whose order complex is a shellable simplicial complex. Finally, in section \ref{sec:5} we conclude our paper and mention an open question that appears to be the missing link to relate the singular homology with the coding theoretic parameters such as generalized rank weights of a rank metric code. 




%% file: Combinatorial.tex
\section{Shellability of order complexes of \texorpdfstring{$q$}{q}-complexes}
In subsection \ref{subsec:2.1}, we recall the relevant combinatorial notions and some results from \cite{GPR21}. The next subsection deals with the shellability of order complexes of shellable $q$-complexes.

\subsection{Preliminaries}\label{subsec:2.1}

Throughout this paper, $q$ denotes a power of a prime number and $\Fq$ the finite field with $q$ elements. We fix a positive integer $n$ and denote by $E$ the $n$-dimensional vector space $\Fqn$ over $\Fq$. Let $\Sigma(E)$ denote the set of all $\Fq$-subspaces of $E$. For $U,V \in \SE$, we write $U \subseteq V $ if $U$ is contained in $V$ and $U \subset V$ if the containment is proper. Since any subspace $U \in \SE$ in this paper is an $\Fq$-subspace, we simply write $\dim \, U$ to mean the $\Fq$-vector space dimension of $U$. We write $\NN$ for the set of all nonnegative integers and $\NN^+$ for the set of all positive integers. For a nonnegative integer $k \leq n$, let ${n \brack k}_q$ denote the $q$-binomial coefficient given by $\frac{\prod_{i=0}^{k-1}(q^n - q^i)}{\prod_{i=0}^{k-1}(q^k-q^i)}$.

We will not review the basic definitions and results concerning simplicial complexes and matroids. Though the $q$-analogs we study in the paper are immensely inspired by these classical notions, we do not need them formally in this work. For a good exposition, one can refer to \cite{Stan} or \cite{GSSV} for simplicial complexes, shellability and related concepts, and to the book \cite{ Oxley,WHITE} for matroids.

\begin{definition}\cite[Section 5]{Rota71}\label{def:q-complex}
 A $q$-complex on $E = \Fqn$ is a set $\Delta$ consisting of subspaces of $E$ satisfying the property that whenever $A \in \Delta$, all the $\Fq$-subspaces of $A$ are in $\Delta$.

Let $\Delta$ be a $q$-complex. The elements of $\Delta$ are called \emph{faces}. The maximal elements of $\Delta$ w.r.t. inclusion are called \emph{facets} of $\Delta$. We say $\Delta$ is \emph{pure} if all the facets have the same dimension. The \emph{dimension} of $\Delta$, denoted as $\dim \, \Delta$, is $\max\{ \dim A \colon A \in \Delta\}.$
\end{definition}

The notion of shellability for $q$-complexes was introduced by Alder \cite[Definition 1.5.1]{Alder}. We recall a different but obviously equivalent definition of shellable $q$-complexes as proved in \cite[Lemma 4.3]{GPR21}.

\begin{definition}\label{def:q-shellable}
    A pure $q$-complex $\Delta$ on $E=\Fq^n$, of dimension $r$ is called shellable if there is a total order $F_1,\dots,F_t$ on the facets of $\Delta$ such that for every $i, j \in \NN^+$
with $i < j \le t$, there exists $k \in\NN^+$ with $k < j$ such that $$F_i\cap F_j\subseteq F_k\cap F_j \text{ and } \dim F_k \cap F_j = r-1.$$ 

\end{definition}


Matroids have been thoroughly studied, for example in~\cite{Oxley}. They have many equivalent definitions (via bases, independent sets, circuits, flats, rank function to name a few). They have natural $q$-analogues (see~\cite{JP18}). 
Here we give the definition of a $q$-matroid via its rank function.
\begin{definition}\label{def:q-matroid}
        A $q$-matroid $M$ over $E =\Fq^n$ is a pair $M=(E,\rho)$ where $\rho$ is a function $\rho:\Sigma(E)\rightarrow \NN$ satisfying
    \begin{enumerate}[(R1)]
        \item $0\leq \rho(U)\leq \dim U$ for all $U\in \Sigma(E)$,
        \item If $U,V\in \Sigma(E)$ with $U\subset V$, then $\rho(U)\leq \rho(V)$,
        \item $\rho(U+V)+\rho(U\cap V)\leq \rho(U)+\rho(V)$ for all $U,V\in \Sigma(E)$.
    \end{enumerate}
\end{definition}

\begin{definition}\label{def:q-indepbase}
        Let $M=(E,\rho)$ be a $q$-matroid. An element $U\in \Sigma(E)$ is said to be an \emph{independent} space of $M$ if $\rho(U)=\dim U$; otherwise, it is called \emph{dependent}. A basis of $M$ is a maximal independent space $U$ of $M$ w.r.t inclusion of subspaces. 
         We write $\I_M$ for the set of independent spaces of the $q$-matroid $M$ and $\B_M$ for the set of all bases of $M$.
\end{definition}

It is well known that the set $\I_M$ of independent spaces 
of a $q$-matroid $M=(E,\rho)$ forms a $q$-complex $\Delta_M$ on $E$, which we refer to as the \emph{$q$-complex associated with $M$}. By a $q$-matroid complex on $E$, we shall mean the $q$-complex associated with a $q$-matroid on $E$. It is easy to see that $\B_M$ is the set of facets of $\Delta_M$ and the $q$-matroid complex $\Delta_M$ is pure. 

In~\cite[Theorem 4.4]{GPR21}, it is proved that any $q$-matroid complex is shellable.
A variant of row reduced echelon forms of subspaces, called \emph{tower decompositions} \cite[Section 3]{GPR21} has been used to order the facets in a shelling order \cite[Lemma 3.6]{GPR21}. Next we give that definition of the ordering in \cite[Definition 3.5]{GPR21} in an alternative form that will be more suitable for this paper. We prove the equivalence of the two orderings in Proposition \ref{prop:equivalent_orderings}.

First, as in~\cite{GPR21}, fix a linear ordering  $\preceq$ on $\F_q$ satisfying $0 \prec 1$ and $(\forall x \in \Fq\backslash\{0,1\})(1  \prec x).$ 
This ordering extends to the lexicographic ordering on $\Fq^n$, which we still denote by $\preceq$, as follows: for $\alpha = (\alpha_1, \ldots, \alpha_n), \beta= (\beta_1, \ldots, \beta_n) \in \Fqn$, if $i$ is the smallest integer such that $\alpha_i\neq \beta_i$, then
\[
\alpha \prec \beta \Leftrightarrow \alpha_i\prec \beta_i.
\]

\begin{definition}\label{def:ordering}
    Let $U\neq V$ be two vector subspaces of $\Fq^n$ of the same dimension. We define a relation $\prec_q$ by \[U \prec_q V \Leftrightarrow \min U \backslash V \prec  \min V \backslash U.\]
\end{definition}
We show that the binary relation $\prec_q$ defines a total ordering on the subspaces of same dimension of $\Fqn$ by proving its equivalence with the ordering defined in \cite[Definition 3.5]{GPR21}. In the proof,  we will use the notation of~\cite{GPR21} without introducing it, except that we denote the ordering in \cite[Definition 3.5]{GPR21} by $\ll$.
\begin{pro}\label{prop:equivalent_orderings}
    The binary relation $\prec_q$ of definition~\ref{def:ordering} is the same as the ordering defined in~\cite[Definition 3.5]{GPR21}. As such, it defines a total ordering on the subspaces of same dimension of $\Fq^n$.
\end{pro}
\begin{proof}
     Since $\ll$ is a total ordering on subspaces of a given dimension, if $U \ll V \Rightarrow U \prec V$ for two subspaces of dimension $r$ of $\Fq^n$, this implies that the two binary relations are the same. So suppose that $U \ll V$. Let $\tau(U)=(U_1,\cdots,U_r)$ and $\tau(V)=(V_1,\cdots,V_r)$ be their tower decompositions. Let $1\leqslant e \leqslant r$ be such that $U_j=V_j$ for any $j <e$ and $U_e \neq V_e$. By hypothesis, we have \[u_e = \min \overline{U}_e \prec   \min \overline{V}_e=v_e.\]
Suppose if possible that $u_e\in V$. 
By~\cite[Lemma 3.3]{GPR21}, we have $p(v_e)\leqslant p(u_e)$. Further, if $p(v_e) = p(u_e)$, then by~\cite[Lemma 3.2]{GPR21}, $u_e \in \overline{V_e}$, which contradicts the minimality of $v_e \in \overline{V_e}$. Thus $p(v_e)<p(u_e)$, and by~\cite[Lemma 3.3]{GPR21} again, this implies that $u_e \in \overline{V_s} = \overline{U_s}$ for some $s<e$, which is absurd. Thus $u_e \notin V$.\\
An easy generalization of~\cite[Lemma 3.4]{GPR21} shows that $\min U\backslash U_{e-1} = \min \overline{U_e}$ and similarly for $V$, so that we get \[u_e = \min \overline{U_e} = \min U\backslash U_{e-1} = \min U \backslash V_{e-1} \prec \min U \backslash V\] but since $u_e \notin V$, there is actually equality. Further, \[\min U \backslash V = u_e \prec v_e = \min \overline{V_e} = \min V \backslash V_{e-1} = \min V \backslash U_{e-1} \prec \min V \backslash U,\] that is \[U \prec_q V.\]
\end{proof}

\begin{rem}
Notice that the ordering $\preceq_q$ can be used only to compare subspaces of the same dimension. For simplicity, we use the same notation for any dimension.
\end{rem}

\begin{thm}\cite[Theorem 4.4]{GPR21}\label{th:shellable}
Let $M=(E,\rho)$ be $q$-matroid of rank $r$. With the ordering on $r$-dimensional subspaces of $E$ from Definition~\ref{def:ordering}, the $q$-complex $\Delta_M$ is shellable.
\end{thm}

Note that the ordering in Definition \ref{def:ordering}
 is essentially a lexicographic ordering, which motivates the following definition. 
 
\begin{definition}
    Let $\Delta$ be a pure $q$-complex on $E=\Fqn$ of dimension $k$. We say that $\Delta$ is a \emph{lexicographically shellable} $q$-complex if the ordering $\preceq_q$ gives a shelling of $\Delta$.
\end{definition}

It is well known that the topological space associated to the (punctured) simplicial complex is homotopy equivalent to a wedge sum of spheres \cite{BH}. Also, their homology groups are determined explicitly  by counting certain facets \cite{BJO92}. An attempt to extend this results to shellable $q$-complexes was made in \cite{GPR21}. There the homology groups of a class of shellable $q$-complexes, including $q$-spheres, uniform $q$-complexes, were described \cite[Theorems 5.7, 5.8, 5.11, 6.10]{GPR21}. In a later section, we will completely describe these homology groups for any lexicographically shellable $q$-matroid complex. In this regard, we prove in the following subsection that the order complexes associated with any lexicographically shellable $q$-complexes are shellable.

\subsection{The order complex of a lexicographically shellable \texorpdfstring{$q$}{q}-complex} \label{subsec:2.2}

The order complexes (formally defined in Definition \ref{def:order-complex}) of a large number of posets in combinatorics are shellable, e.g., bounded, locally upper semimodular posets. In literature, these posets are called \emph{shellable posets}. More finer notions of shellable posets were introduced by Bj{\"o}rner, which are called \emph{EL(edgewise lexicpgraphic)-shellable} and \emph{CL(chainwise lexicogrpahic)-shellable} posets \cite{bjornerlexi, Bjorner}. In fact, it is proved in \cite[Proposition 2.3]{bjornerlexi}, that EL-shellabile $\Rightarrow$ CL-shellable $\Rightarrow$ shellable. However, in this section we  prove that order complexes of lexicographically shellable $q$-complexes are shellable, without explicitly using the methods of $EL$-labeling or $CL$-labeling of \cite{bjornerlexi}.

\begin{definition}[Order complexes]
\label{def:order-complex}
Let $\Delta$ be a $q$-complex over $\Fq^n$. A chain in $\Delta$ is a subset $\UU\subset \Delta$ whose elements are linearly ordered under the subset relation. We abbreviate such a chain with the notation $\UU=\{U_{i_1}\subset\dots \subset U_{i_l}\}$. 
In addition, the empty set is also considered as a chain in $\Delta$. The order complex $K(\Delta)$ associated to $\Delta$ is the set of all chains in $\Delta$. 
\end{definition}

The order complex associated to a $q$-complex is a simplicial complex. From an ordering of the facets of the original $q$-complex $\Delta$, one can associate an ordering on the facets of the order complex $K(\Delta)$ as follows.

\begin{definition}[A lexicographic order on chains]
Let $K(\Delta)$ be the order complex associated to a $q$-complex $\Delta$. Suppose that the facets of $\Delta$ of the same ranks are ordered with $\preceq_q$. If $\UU=\chain{U}{r}$ and $\VV=\chain{V}{r}$ are two maximal chains in $K(\Delta)$, we say that $U\preceq_l V$ if one of the following is true.
\begin{enumerate}[(i)]
    \item $\UU=\VV$,
    \item $e$ is the largest index such that $U_e\neq V_e$ and $U_e\prec_q V_e$.
\end{enumerate}
\end{definition}

\begin{pro}
Let $\Delta$ be a $q$-complex over $\Fq^n$ and let $C_r$ be the set of all chains of the form $\chain{U}{r}$, where $\dim U_{i}=i$. Then $\preceq_l$ is a total order.
\end{pro}
\begin{proof}\
\begin{itemize}
    \item Comparability (For any chains $\UU,\VV$, either $\UU\preceq_l \VV$ or $\VV\preceq_l \UU$): If $\UU=\VV$ then it's clear. If $e$ is the largest integer such that $U_e\neq V_e$ then either $U_e\pq V_e$ or $V_e\pq U_e$.
    \item Antisymmetry (If $\UU\preceq_l \VV$ and $\VV \preceq_l \UU$, then $\UU=\VV$): Suppose that $\UU\neq \VV$, then let $e$ be the largest index where $U_e\neq V_e$. Then, we can only have one of  $U_e\pq V_e$ and $V_e\pq U_e$. Hence we can only have one of $\UU\lo \VV$ and $\VV \lo \UU$.
    \item Transitivity (If $\UU\preceq_l \VV$ and $\VV \preceq_l \WW$, then $\UU\preceq_l \WW$): We may assume that $\UU\neq \VV$ and $\VV\neq \WW$. Let $e$ be the largest index such that $U_e\neq  V_e$ and $U_e\pq V_e$. Similarly, let $f$ be the largest index such that $V_f\neq  W_f$ and $V_f\pq W_f$.
    \begin{itemize}
        \item Case 1: $e> f$. Then $e$ is the largest index such that $U_e\neq W_e$. And $U_e\pq V_e = W_e$ and thus $U\lo W$.
        \item Case 2: $e=f$. Then $e$ is the largest index such that $U_e\neq W_e$. And $U_e\pq V_e \pq W_e$ and thus $U\lo W$.
        \item Case 3: $e<f$. Then $f$ is the largest index such that $U_f\neq W_f$ and $U_f = V_f\pq W_f$.
    \end{itemize}
\end{itemize}
\end{proof}

\begin{definition}
Let $\UU= \chain{U}{r}$ be a chain with $\dim U_i=i$ and let $A$ be a subspace of dimension $i$ satisfying $U_{i-1}\subset A \subset U_{i+1}$. Define the chain ${\UU}^{A,i}$ as 
\[
{\UU}^{A,i} = \{ U_0\subset \dots \subset U_{i-1} \subset A \subset U_{i+1} \subset \dots \subset U_r\}.
\]
\end{definition}

\begin{lem}\label{lem:2}
Let $\UU= \chain{U}{r}$ be a chain and let $A$ be a subspace of dimension $i$ satisfying $U_{i-1}\subset A \subset U_{i+1}$. If $A\prec_q U_i$ , then
\[
\#(\UU\cap {\UU}^{A,i}) = r-1 \text{ and } {\UU}^{A,i} \lo \UU.
\]
\end{lem}
\begin{proof}
This directly follows from the definition of $\prec_l$.
\end{proof}

\begin{lem}\label{lem:3}
Let $s$ and $s_1$ be two integers such that $s_1\leq s-2$. 
Let $\VV=\{V_{s_1}\subset V_{s_1+1}\subset \dots \subset V_s\}$ be a chain such that for all $i$ with $s_1<i< s$, $V_i = \min_{\preceq_q}\{F:V_{i-1}\subset F\subset V_{i+1}\}$. 
Then for all $i$ with $s_1<i< s$, $V_{i} = V_{i-1}\oplus \< u_{i} \>$ and $u_{s_1+1} \prec u_{s_1+2} \prec \cdots \prec u_s$ where $u_i = \min V_s\backslash V_{i-1} = \min V_{i}\backslash V_{i-1}$. 
\end{lem}

\begin{proof}
Let $u_i = \min V_s\backslash V_{i-1}$ for $s_1<i< s$. Then we claim that $V_i := \min_{\preceq_q}\{F:V_{i-1}\subset F\subset V_{i+1}\} = V_{i-1}\oplus \< u_i \>$.

\textbf{Proof of the claim}:
Let $W = V_{i-1}\oplus \< u_i \>$ and suppose that $W \neq V_i$. By the definition of $V_i$, $V_i \prec_q W$. Thus, 
$\min V_i \backslash W  \prec \min W \backslash V_i$. Note that $V_i \backslash W \subseteq V_{i+1} \backslash V_{i-1} \subseteq V_s \backslash V_{i-1}$. Hence $u_i=\min V_s\backslash V_{i-1}\preceq \min V_i\backslash W\prec \min W\backslash V_i$.
On the other hand, since $u_i \notin V_i$, then $u_i \in W \backslash V_i$. Hence $\min W \backslash V_i\preceq u_i$, which gives a contradiction. This proves the claim.

Thus it is proved that $V_i = \min_{\pq}\{F:V_{i-1}\subset F\subset V_{i+1}\} = V_{i-1}\oplus \< u_i \>$, where $u_i = \min V_s\backslash V_{i-1}$ for $s_1<i< s$. Now, note that since $u_i \in V_i$, minimality of $u_i$ in $V_s \backslash V_{i-1}$ implies that $u_i = \min V_i \backslash V_{i-1}$.

On the other hand, since $u_i = \min V_s \backslash V_{i-1}$ for all $s_1 < i \leq s$, it immediately follows that $u_{s_1+1} \prec u_{s_1+2} \prec \cdots \prec u_s$. Thus it completes the proof of the lemma.
\end{proof}

We show that the order complex of a $q$-sphere is a  shellable simplicial complex.

\begin{thm}[Shellability of the order complex of a $q$-sphere]\label{thm:2.10}
Suppose that $U_{r+1}$ is a vector space of dimension $r+1$. Let $S_{q}^{r}$ be the set of all codimension $1$ subspaces of $U_{r+1}$ and let $K(S_{q}^{r})$ be the corresponding order complex. Then $\preceq_l$ defines a shelling on $K(S_{q}^{r})$.
\end{thm}
\begin{proof}
Let $\UU \neq \VV$ be two maximal chains in $K(S_{q}^{r})$ such that $\UU\lo \VV$
 where 
 \[
 \UU=\chain{U}{r} \text{ and } \VV=\chain{V}{r}.
 \]
 Let $s$ be the largest index where $U_s\neq V_s$. Therefore $U_s\prec_q V_s$. If necessary, we add $U_{r+1}$ to both sequences, and therefore, without loss of generality, we may assume that $U_{s+1} = V_{s+1}$. Let $a$ be the largest index such that $a < s$ and $U_{a} = V_{a}$. Note that $a$ must exist since $U_0 = V_0$. If $a = s-1$, then we take the chain $\WW=\VV^{U_{s},s}$. It is clear that $\UU\cap \VV \subset \WW\cap \VV$ and by Lemma \ref{lem:2}, $\#(\WW\cap \VV) = r-1$ and $\WW \lo \VV$ and we are done. 

Next, we assume that $a\leq s-2$. Suppose that  
 $V_i = \min_{\preceq_q} \{F \colon V_{i-1} \subset F \subset V_{i+1}\}$ for all $a < i \leq s$.

Let $j$ be the largest integer such that $a\leq j<s$ and that $V_j\subset U_s$. Such $j$ exists since $V_a=U_a\subset U_s$. Since $U_s\subset V_{s+1}$, then  $U_s\backslash V_s\subset V_{s+1}\backslash V_s$. Therefore
\begin{equation}\label{eq:1}
\min V_{s+1}\backslash V_s\preceq \min U_s\backslash V_s.
\end{equation}
Now $V_j\subset V_s$ and this implies that $V_{s+1}\backslash V_s\subset V_{s+1}\backslash V_j$. Therefore
\begin{equation}\label{eq:2}
   \min V_{s+1}\backslash V_j \preceq  \min V_{s+1}\backslash V_s.
\end{equation}
Equations \eqref{eq:1} and \eqref{eq:2} implies that 
\begin{equation}\label{eq:3}
\min V_{s+1}\backslash V_j \preceq \min U_s\backslash V_s.
\end{equation}
Since $U_s\prec_q V_s$, then by definition, $\min U_s\backslash V_s \prec \min V_s\backslash U_s$. This together with equation \eqref{eq:3} implies that 
\begin{equation}\label{eq:4}
    \min V_{s+1}\backslash V_j\prec \min V_s\backslash U_s.
\end{equation}
Now, from Lemma \ref{lem:3}, let $v_{j+1} := \min V_{s+1}\backslash V_j = \min V_{j+1}\backslash V_j$. Since $v_{j+1}\in V_{j+1}$ and $V_{j+1}\subset V_s$, then $v_{j+1}\in V_s$. If $v_{j+1}\notin U_s$, then $v_{j+1}\in V_s\backslash U_s$ and therefore  $\min V_s\backslash U_s\preceq v_{j+1}$. 
By definition of $v_{j+1}$, this means that 
\begin{equation}\label{eq:5}
\min V_s\backslash U_s\preceq \min V_{s+1}\backslash V_j.
\end{equation}
Equations \eqref{eq:4} and \eqref{eq:5} implies $\min V_{s+1}\backslash V_j\prec \min V_{s+1}\backslash V_j$ which is impossible. Hence, we must have $v_{j+1}\in U_s$. That implies $V_{j+1}\subset U_s$. By our hypothesis, $j$ is the largest integer such that $a\leq j<s$ and $V_j\subset U_s$. Therefore $j+1\geq s$. Therefore $V_s \subset V_{j+1}\subset U_s$. Since $V_s$ and $U_s$ have the same dimension, then $U_s = V_s$ which is again a contradiction.

 Hence, for some $a < i \leq s$, $V_i \neq \min_{\preceq_q} \{F \colon V_{i-1} \subset F \subset V_{i+1}\}$. Then take $W_i = \min_{\preceq_q} \{F \colon V_{i-1} \subset F \subset V_{i+1}\}$. Define the chain $\WW = \VV^{W_i,i}$ we see that $\UU\cap \VV \subset \WW\cap \VV$ and by Lemma \ref{lem:2}, $\#(\WW\cap \VV) = r-1$ and $\WW \lo \VV$ and we are done.
\end{proof}


In the previous theorem, we have a statement written using the term shellability of the order complex of a $q$-sphere. In the following corollary, we rewrite it explicitly in terms of properties of chains to make it easy to use it in the next theorem.


\begin{cor}\label{cor:4}
Let $\UU=\chain{U}{r}$ and  $\VV=\chain{V}{r}$ be two chains such that $U_r$ and $V_r$ are of dimension $r$ and they are contained in a common subspace $V_{r+1}$ of dimension $r+1$. Then there exists a subspace $A_i$ of dimension $i\leq r$ such that $\VV^{A_i,i} \lo \VV$ and $\UU\cap \VV \subset \VV^{A_i,i}\cap \VV$.
\end{cor}

\begin{thm}\label{thm:2.12}
Let $\Delta$ be a lexicographically shellable $q$-complex . The ordering $\preceq_l$ is a shelling on the order complex $K(\Delta)$.
\end{thm}
\begin{proof}
Let $\UU=\chain{U}{r}$ and $\VV=\chain{V}{r}$ be two distinct chains. Let $e$ be the largest index such that $U_e\neq V_e$.

\noindent\textbf{Case 1}: Suppose that $e<r$ and thus there is $U_{e+1} = V_{e+1}$. Let $f$ be the largest index smaller than $e$ such that $U_f = V_f$. Such $f$ exists because $U_0 = V_0$. Consider the two chains
\begin{align*}
    &\overline{\UU}=\{V_0\subset \dots \subset V_{f}\subset U_{f+1} \subset \dots U_{e}\}\\
    &\overline{\VV}=\{V_0\subset \dots \subset V_{f}\subset V_{f+1} \subset \dots V_{e}\} \\
\end{align*}
We see that $\overline{U}_{e}$ and $\overline{V}_{e}$ are both of dimension $e$ and they are contained in $U_{e+1}$ which is of dimension $e+1$. Thus by the Corollary \ref{cor:4}, there exists a subspace $A_i$ of dimension $i$ such that $f<i\leq e$, $\overline{\VV}^{A_i,i} \lo \overline{\VV}$ and $\overline{\UU}\cap \overline{\VV} \subset \overline{\VV}^{A_i,i}\cap \overline{\VV}$. This implies that $\VV^{A_i,i}\lo \VV$ and $\UU\cap \overline{\VV} \subset \VV^{A_i,i}\cap \VV$ and we are done.

\noindent\textbf{Case 2}: Suppose that $e=r$ so that $U_r\pq V_r$. Let $f$ be the largest index such that $U_f=V_f$. Again $f$ exists because $U_0 = V_0$. Now, since $\Delta$ is lexicographically shellable, there is $W_r$ such that $W_r\pq V_r$, $W_r\cap V_r$ is of dimension $r-1$ and $U_r\cap V_r\subseteq W_r\cap V_r$. Let $t$ be the largest index such that $V_t\subset W_r$. Then $t\geq f$. Now, take the chains
\begin{align*}
    &\WW=\{V_0\subset \dots \subset V_{f}\subset \dots \subset V_{t}\subset W_{t+1}\subset \dots \subset W_{r}\},\\
    &\VV=\{V_0\subset \dots \subset V_{f}\subset \dots \subset V_{t}\subset V_{t+1}\subset \dots \subset V_{r}\}.\\
\end{align*}
Since $\dim W_r = \dim V_r = r$ and $\dim W_r\cap V_r = r-1$, then $\dim W_r+V_r=r+1$, i.e., $W_r$ and $V_r$ are contained in a subspace $W_r+V_r$ of dimension $r+1$. Thus we can apply Corollary \ref{cor:4} to deduce that there exists a subspace $A_i$ of dimension $i$ such that $t<i\leq r$, $\VV^{A_i,i} \lo \VV$ and $\WW\cap \VV \subset \VV^{A_i,i}\cap \VV$. Since $\UU\cap \VV\subset \WW\cap \VV$, then we are done.
\end{proof}

\begin{rem}
Notice that in order to apply Corollary \ref{cor:4} (which is essentially Theorem \ref{thm:2.10}), we need to use the ordering $\preceq_q$. Due to this, Theorem \ref{thm:2.12} works only when the shelling of the $q$-complex is given by $\preceq_q$ i.e the $q$-complex is lexicographically shellable. It is still an open question whether the statement holds for any shelling of the $q$-complex.
\end{rem}

At this point one might wonder if the order complex of a lexicographically shellable $q$-complex, which is proved to be shellable, is a matroid complex. We show next that it is not true almost all the time, even if the $q$-complex is not shellable. 

\begin{pro}
Let $\Delta$ be a nontrivial $q$-matroid complex on $\Fq^n$ of dimension $d \geq 2$. Then its order complex $K(\Delta)$ is not a matroid complex.
\end{pro}
\begin{proof}
 Let $E$ be the ground set of $K(\Delta)$, i.e. $E$ consists of all the faces of $\Delta$. If the order complex is a matroid complex, then for any subset $A \subseteq E$, the restricted complex $K(\Delta)|_{A}$ should be a pure simplicial complex. This is an easy consequence of the augmentation axiom of the independent sets of matroids (see, e.g., \cite[Exercise 7.3.1]{BJO92}). Since $\Delta$ is nontrivial, it has at least two distinct facets of dimension $d$, and therefore, there exists two distinct $2$-dimensional faces of $\Delta$, say, $F_1, F_2$. Since $F_1, F_2$ are distinct, there exist $x_1,x_2 \in \Fq^n$ such that $x_1 \in F_1 \backslash F_2$ and $x_2 \in F_2 \backslash F_1$. Now if we take $A = \{ \0, \langle x_1 \rangle, \langle x_2 \rangle, F_1 \}$, then $\0 \subseteq \langle x_1 \rangle \subseteq F_1$ and $\0 \subseteq \langle x_2 \rangle$ are two maximal chains of different lengths. Thus the restricted complex is not pure and hence we can conclude that the order complex is not a matroid complex.
\end{proof}

\begin{rem}\label{CMposets}
    We note that Theorem \ref{thm:2.12} implies that the Stanley-Reisner ring corresponding to the order complex of a lexicographically shellable simplicial complex is Cohen-Macaulay (over any field). Thus the posets underlying the lexicographically shellable $q$-complexes are Cohen-Macaulay posets. The Cohen-Macaulay posets were introduced by Baclawski \cite[Section 3]{Baclawski} in purely combinatorial way, whereas Reisner \cite{Reisner} and Stanley \cite{Stanley} independently gave a ring theoretic characterization of Cohen-Macaulay posets, with Reisner proving their equivalence in \cite{Reisner}. We leave open the question whether the underlying poset of any shellable $q$-complex is Cohen-Macaulay poset.
\end{rem}

%% file: Topology.tex
\section{The homotopy type of a pure shellable \texorpdfstring{$q$}{q}-complex}\label{Sec:homotopy}

From a q-complex, we can build three different kinds of spaces from it. 
\begin{enumerate}
    \item A q-complex $\Delta$ is a partially ordered set that is also finite, therefore $\Delta$ determines a finite topological space by equipping the set $\Delta$ with the order topology. We write $\TT(\Delta)$ for this space. The space $\TT(\Delta)$ is in fact an Alexandroff $T_0$-space. 
    \item Every partially ordered set determines a simplicial set via the nerve construction. For a q-complex $\Delta$, we write $N(\Delta)$ for its nerve. The sets $N(\Delta)_i$ consist of the chains of subspaces of length $i$, that is, $V_0 \subseteq V_1 \subseteq \cdots V_i$. The face and degeneracy maps come from composition of inclusions (removing a subspace) or insertion of an equality (duplicate a subspace in the chain).

    \item Every simplicial set has a geometric realization, so a q-complex $\Delta$ also determines a topological space  $B(\Delta) = \lvert N(\Delta) \rvert$.
\end{enumerate}

\begin{rem}
    All of the above constructions yield contractible spaces due to the presence of the minimal element $\{0\} \in \Delta$ for every q-complex. It is therefore more useful for us to study the above constructions on the poset $\Deltao = \Delta \setminus \{ \{0\} \}$ where the zero subspace has been removed. We will therefore study the spaces $\TT(\Deltao)$, $N(\Deltao)$, and $B(\Deltao)$ instead.
\end{rem}

\begin{rem}
    The order complex $K(\Delta)$ of definition \ref{def:order-complex} is an abstract simplicial complex, not a simplicial set. However, the geometric realizations of $K(\Delta)$ and $N(\Delta)$ are homeomorphic, and we may therefore study the homotopy type of $N(\Delta)$ in the more suitable homotopy category of simplicial sets. A rough explanation for why these spaces are homeomorphic is that  the order complex $K(\Delta)$ records the non-degenerate simplices of $N(\Delta)$, and the geometric realization of $N(\Delta)$ is built out of the non-degenerate simplices. See \cite[\S8.1, Exercise 8.1.4]{Hbook} for details. 
\end{rem}

The partially ordered set $\Deltao$ has three different kinds of spaces associated to it, which live in three different model categories. 
In what sense are these constructions equivalent or different? 
Our motivation from coding theory motivates us to understand the homology groups of these spaces. We can in fact do better and understand the (weak) homotopy type of these spaces. Since integral homology is a weak homotopy invariant,  knowing the homotopy type of these spaces determines their homology \cite[Proposition 4.21]{Hatcher}. 

\subsection{Background on topology}

The basic theory of topological spaces, homotopy, and homology can be found in the books \cite{Sieradski}, \cite{Switzer}, and \cite{Hatcher}. For rhetorical purposes, we will emphasize the key ideas and constructions that will be used in the argument of this section. 

If $f_0: X \to Y$ and $f_1 :X\to Y$ are continuous maps between topological spaces, a \emph{homotopy} between the maps $f_0$ and $f_1$ is a continuous map $ H : X \times [0,1] \to Y$ that satisfies $f_0(x) = H(x,0)$ for all $x \in X$ and $f_1(x) = H(x,1)$ for all $x \in X$. Intuitively, a homotopy gives a continuous deformation of the function $f_0$ into $f_1$. 

Two topological spaces $X$ and $Y$ are \emph{homeomorphic} if there exists an invertible, continuous function $f : X \to Y$ whose inverse $f^{-1} : Y \to X$ is also continuous. The idea of \emph{homotopy equivalence} is to weaken the notion of homeomorphism, so that a map $f :X \to Y$ is not strictly invertible, but instead, invertible ``up to homotopy.'' A continuous function $f: X \to Y$ is a \emph{homotopy equivalence} if there exists a continuous function $g : X \to Y$ such that $g\circ f$ is homotopic to $\id_X$ and $f \circ g$ is homotopic to $\id_Y$. The notion of homotopy equivalence acts as an equivalence relation on the category of topological spaces. As such, the category of topological spaces is partitioned up into equivalence classes for the homotopy equivalence relation. The strict homotopy type of a space is the equivalence class of the space under the homotopy equivalence relation.

The notion of strict homotopy equivalence is a bit too strict. The notion of weak homotopy equivalence and weak homotopy type turns out to better reflect our intuition in homotopy theory. A map $f : X \to Y$ is a \emph{weak homotopy equivalence} if the induced map $f_* : \pi_n X \to \pi_n Y$ is an isomorphism for all natural numbers $n$. (We gloss over the discussion on basepoints here). The notion of weak homotopy equivalence again acts as an equivalence relation on topological spaces, and to understand the (weak) homotopy type of a topological space is to determine its equivalence class for the relation of weak homotopy equivalence. 

The $CW$-approximation theorem \cite[Proposition 4.13]{Hatcher} says that every topological space has the weak homotopy type of a $CW$-complex. Note that $CW$ complexes are especially nice kinds of spaces, built up inductively by gluing together euclidean balls of increasing dimensions. The benefit of this theorem is that we can always describe the weak homotopy type of a space with a $CW$-complex, and the simplest kinds of $CW$-complexes are spheres $S^n$ and wedge-sums of spheres---much like $\bbZ$ and finite direct sums of $\bbZ$ are the simplest kinds of abelian groups up to isomorphism.

\subsection{Relevant constructions interpreted with model categories}

Recall that to a $q$-complex $\Delta$, we have defined three different kinds of spaces: $\TT(\Deltao)$, $N(\Deltao)$, and $B(\Deltao)$. These three spaces live in different categories, but are all equivalent in a precise sense.  
With the framework of model categories, it has been shown that the three relevant model categories these spaces live in are all equivalent.
The framework of model categories and the equivalences between the categories serve as a dictionary to translate between the different spaces and their properties. 
We summarize the key results now. 

The category of posets $\Pos$ is isomorphic to the category of $T_0$ Alexandroff spaces $\calA$ \cite[Proposition 4.2]{Raptis}.
The isomorphism is given by the functors $\TT : \Pos \to \calA$ and $P : \calA \to \Pos$, which are defined in \cite[Proposition 4.2]{McCord}. 
There are thus two evident choices of weak equivalences in the categories $\calA \cong \Pos$. One choice, when viewing $X, Y\in\calA$ as topological spaces is to use the usual notion of weak homotopy equivalence, essentially, $f : X \to Y$ is a weak homotopy equivalence if it induces isomorphisms in all homotopy groups. 
The other notion is to use the nerve functor $N : \Pos \to \Sset$ to define weak equivalences. A morphism $f : C \to D$ in $\Pos$ is a weak equivalence if the corresponding map of simplicial sets $N (f) : N(C) \to N(D)$ is a weak equivalence of simplicial sets.
McCord has shown that these two notions of weak equivalence coincide \cite[Theorem 4.5]{Raptis}, \cite{McCord}.
As a consequence, every $A$-space $X \in \calA$ is weak homotopy equivalent to its classifying space $B X$. 
Raptis has also developed a model category structure on $\Pos$ in \cite[Theorem 2.6]{Raptis} that fits into a zig-zag of Quillen equivalences $\Pos \to \Sset \leftarrow \Top$, enriching McCord's result. 
\begin{equation*}
\label{eq:quillen_equivalences}
\xymatrix{
\calA \ar@{{<}{-}{>}}[r]^{P}_{\TT} & \Pos \ar[r]^i & \Cat \ar[r]^{\Ex^2 N}& \Sset & \Top \ar[l]_{S_*}
}
\end{equation*}
The above diagram appeared in \cite[p 223]{Raptis}. Here $S_* : \Top \to \Sset$ is the singular chains functor. Note that the functor $\Cat \to \Sset$ is not just the nerve functor, but rather $\Ex^2 N$. Since $\Ex : \Sset \to \Sset$ is a Quillen equivalence too, the natural map $N(C) \to \Ex^2 N(C)$ is a weak equivalence for any small category (or poset) $C \in \Cat$.  The use of $\Ex^2$ is necessary for certain properties of a Quillen equivalence to hold. See Thomason's paper where the model category structure on $\Cat$ is introduced \cite{Thomason}.  

We show these equivalences to suggest how to work and think about the various spaces that can be constructed out of a q-complex. All of the constructions land in one of these Quillen-equivalent categories, and at the level of homotopy categories, all of the constructions correspond to one another up to homotopy via the categorical equivalences in the diagram below.

\begin{equation}
\label{eq:ho_cats}
\xymatrix{
\calH(\calA) \ar[r]^{P} & \calH(\Cat) \ar[r]^{N} & \calH(\Sset) \ar[r]_{\lvert - \rvert} & \calH(\Top) \ar[l]_{S_*}
}
\end{equation}

As a practical matter, we can think of a sphere $S^n$ in any one of these categories by using these equivalences. The simplest illustration of this is that the standard $n$-simplex $\Delta^n$ is a simplicial set whose boundary $\partial \Delta^n$ models the sphere $S^{n-1}$ in the category of simplicial sets. It may be readily verified that indeed the geometric realization of $\partial \Delta^n$ is weak equivalent to 
the euclidean sphere $S^{n-1}$ in $\Top$. Since $\partial \Delta^n$ is a simplicial set model for the homotopy type of the sphere $S^{n-1}$, and the homotopy categories $\calH(\Sset)$ and $\calH(\Top)$ are isomorphic, we may abuse notation and write $S^{n-1}$ for the homotopy type of $\partial \Delta^n$ in $\Sset$. 

\subsection{Proof of the result}

Our main theorem will establish the homotopy type of the spaces associated to a pure, shellable, q-complex, which matches well with the corresponding result about shellable simplicial complexes. The proof given below can be easily simplified to give a proof in the simplicial complex case. 

\begin{thm}
\label{thm:main}
Let $\Delta$ be a pure, shellable, q-complex of dimension $k$ that arises from an $n$-dimensional vector space over a finite field $F = \bbF_q$. The homotopy type of the spaces $\TT(\Deltao)$, $N(\Deltao)$, and $B(\Deltao)$ is that of a wedge sum of spheres of dimension $k-1$. 
\end{thm}

We prove the result for the spaces $N(\Deltao)$, as the homotopy type of the other spaces is then determined from the properties of the equivalences in diagram \ref{eq:ho_cats}. 
The proof is broken into two steps: (1) $\Delta$ is a codimension 1 pure, shellable, q-complex in $\bbF_q^n$, induction on $n$; (2) $\Delta$ is a pure, shellable, q-complex of dimension $k$, induction on $k$. 

The key idea of the proof is that by focusing only on the weak homotopy type of the spaces, we can exploit the shellability property by using homotopy pushouts to identify the correct homotopy type of the resulting space at each stage of the construction given by the shelling. The interested reader can learn more about homotopy pushouts and homotopy colimits in \cite{BousfieldKan} and \cite{MayPonto}. The style of argument is very similar to the way one can compute the homology of these spaces with the Mayer-Vietoris sequence, we just have to be more careful about the details in finding the correct homotopy type of the result of the gluings using homotopy colimits. 

\begin{lem}
\label{lem:codim1-shellable}
Every pure q-complex $\Delta$ in $\bbF_q^n$ of dimension $n-1$ is shellable.  
\end{lem}

\begin{proof}
The $q$-complex $\Delta$ is generated by a finite list of $n-1$ dimensional subspaces of $\bbF_q^n$, which we enumerate as $F_1,\ldots, F_k$.
Any ordering of the facets $F_1,\ldots, F_k$ is a shelling of $\Delta$. The key observation is that $\langle F_1,\ldots F_i\rangle \cap \langle F_{i+1} \rangle$ is generated by the subspaces $F_j \cap F_{i+1}$ for $1\leq j \leq i$, all of which have dimension $n-2$ since $F_j \neq F_{i+1}$ are distinct codimension 1 subspaces of $\bbF_q^n$. 
\end{proof}

\begin{lem}
\label{lem:codim1}
Let $\Delta$ be a pure q-complex of codimension 1 in $\bbF_q^n$. The homotopy type of $N(\Deltao)$ is that of a wedge sum of spheres of dimension $n-2$. 
\end{lem}

\begin{proof}
To avoid trivialities, consider $n = 2$ as the base case, in which case $\Deltao$ is just a set of $1$-dimensional subspaces of $\bbF_q^2$. No two distinct elements of $\Deltao$ are comparable, hence $N(\Deltao)$ has the homotopy type of a finite set of points, i.e., a wedge sum of $0$-dimensional spheres. 

Now make the induction hypothesis, with $n \geq 2$, that any pure, codimension 1 q-complex $\Delta$ of $\bbF_q^n$ has the homotopy type of a wedge of $n-2$-dimensional spheres $\bigvee_{I} S^{n-2}$ where $I$ is some finite indexing set depending on $\Delta$. 

Consider now $\Delta$ a pure, codimension 1 q-complex of $\bbF_q^{n+1}$. By lemma \ref{lem:codim1-shellable}, any ordering of the facets of $\Delta$ is a shelling. Write $F_1,\ldots, F_k$ for some ordering of the facets of $\Delta$. We build $N(\Deltao)$ one facet at a time and determine its homotopy type at each stage. Write $\langle F_1, \ldots, F_j \rangle$ for the nerve of the q-complex generated by the facets $F_1,\ldots F_j$, but with the 0-dimensional subspace removed. Note all such simplicial sets are subcomplexes of $N(\Deltao)$. 

To start, $\langle F_1 \rangle$ is contractible since the partially ordered set defining it contains a terminal element $F_1$ \cite[IV Example 3.2.2]{Kbook}. Suppose now as an induction hypothesis that $\langle F_1, \ldots F_j \rangle$ has the homotopy type of a wedge of $n-1$ dimensional spheres. We now see what the effect of adding the facet $F_{j+1}$ is on the homotopy type of our space. The diagram below is a homotopy push-out square, since the upper row and left column are cofibrations \cite[Ch. XII, \S3, 3.1 Examples]{BousfieldKan}. 
\begin{equation*}
  \xymatrix{
    \langle F_1,\ldots F_j \rangle \cap \langle F_{j+1} \rangle \ar[d] \ar[r] & \langle F_1, \ldots, F_j\rangle \ar[d] \\
    \langle F_{j+1} \rangle \ar[r] & \langle F_1, \ldots F_{j+1} \rangle \\
  }
\end{equation*}
Thus the space $\langle F_1, \ldots, F_{j+1} \rangle$ models the homotopy push-out of $\langle F_{j+1} \rangle \leftarrow \langle F_1, \ldots F_j \rangle \cap \langle F_{j+1} \rangle \to \langle F_1, \ldots F_j \rangle$. Our induction hypothesis allows us to compute this homotopy push-out with the equivalent diagram
\begin{equation*}
\xymatrix{
\langle F_{j+1} \rangle \ar[d]^{\simeq}
& \langle F_1, \ldots, F_j \rangle \cap \langle F_{j+1} \rangle \ar[d]^{\simeq} \ar[r]  \ar[l]
& \langle F_1, \ldots, F_j \rangle \ar[d]^{\simeq}
\\
\text{pt}
& \bigvee_I S^{n-2} \ar[l] \ar[r]
& \bigvee_J S^{n-1} \\
}
\end{equation*}
The left hand homotopy equivalence is due to $\langle F_{j+1} \rangle$ containing a terminal element, and so is contractible. The right most homotopy equivalence comes from our induction hypothesis. 
The middle homotopy equivalence comes from the observation that $\langle F_1, \ldots F_j \rangle \cap \langle F_{j+1}\rangle$ is a subcomplex of $\langle F_{j+1} \rangle$, which arises from a pure q-complex of codimension 1 in the $\bbF_q$-vector space $F_{j+1}$ of dimension $n$ by the defining property of a shelling (as given in \cite[Definition 1.5.1.]{Alder}).
As such, our induction hypothesis applies to show that $\langle F_1, \ldots F_j \rangle \cap \langle F_{j+1} \rangle$ has the homotopy type of a wedge sum of spheres of dimension $n-2$. 

To conclude, the homotopy push-out of the diagram 
\begin{equation*}
\xymatrix{
\text{pt}
& \bigvee_I S^{n-2} \ar[l] \ar[r]
& \bigvee_J S^{n-1} \\
}
\end{equation*}
is simply the mapping cone of the map $\bigvee_I S^{n-2} \to \bigvee_J S^{n-1}$. Since the homotopy group $\pi_{n-2} (\bigvee_J S^{n-1})$ is trivial for all $n \geq 2$, the map is homotopic to a constant map. Therefore the mapping cone is simply $\bigvee_J S^{n-1} \vee \Sigma \bigvee_I S^{n-2} \simeq \bigvee_J S^{n-1} \vee \bigvee_I S^{n-1}$, where $\Sigma X$ denotes the reduced suspension of a space $X$. The result now follows. 
\end{proof}

We can now prove the main result with essentially the same argument. 

\begin{proof}[Proof of Theorem \ref{thm:main}]
We run induction on the dimension $k$ of the pure, shellable, q-complex $\Delta$ in $\bbF_q^n$. The case $k=1$ is immediate, as $N(\Deltao)$ is just a finite set of points. Assume that all pure, shellable q-complexes of dimension $k$ have the homotopy type of a wedge of $k-1$-dimensional spheres, with $k \geq 1$. Consider now a pure, shellable, q-complex of dimension $k+1$. We are given an ordering of its facets $F_1, \ldots F_\ell$. We inductively add facets one at a time and keep track of how the homotopy type of its nerve changes.

We start with $\langle F_1 \rangle$, which is a contractible simplicial set as its defining poset contains $F_1$ as a maximal element. Suppose now for an induction argument that $\langle F_1, \ldots , F_j \rangle$ has the homotopy type of a wedge sum of spheres $\bigvee_I S^{k}$. Adjoining one more facet yields the following homotopy pushout diagram
\begin{equation*}
  \xymatrix{
    \langle F_1,\ldots F_j \rangle \cap \langle F_{j+1} \rangle \ar[d] \ar[r] & \langle F_1, \ldots, F_j\rangle \ar[d] \\
    \langle F_{j+1} \rangle \ar[r] & \langle F_1, \ldots F_{j+1} \rangle \\
  }
\end{equation*}
The key observation is that $\langle F_1,\ldots F_j \rangle \cap \langle F_{j+1} \rangle$ comes from a pure $q$-complex of codimension 1 in the vector space $F_{j+1}$ of dimension $k$ by the shellability property. Hence lemma \ref{lem:codim1} shows that this space has the homotopy type $\bigvee_J S^{k-1}$. The same considerations as in lemma \ref{lem:codim1} allow us to identify this homotopy push-out as the mapping cone of the homotopically trivial map $\bigvee_J S^{k-1}  \to \bigvee_I S^k$, which gives the result. 
\end{proof}

\begin{cor}
If $\Delta$ is a pure, shellable $q$-complex of dimension $k$, then the reduced integral homology of the spaces $\TT(\Deltao)$, $N(\Deltao)$, $B(\Deltao)$, $K(\Deltao)$ is concentrated in a single degree. 
\end{cor}
\begin{proof}
As all of the spaces are a finite wedge sum of spheres of the same dimension, the wedge axiom and dimension axiom for Eilenberg--Steenrod homology theories yield the result \cite[\S2.3]{Hatcher}. 
\end{proof}

\begin{rem}
The proof of theorem \ref{thm:main} gives an algorithm for computing the homology of a pure shellable $q$-complex. One recursively identifies the homology groups $\redH_*(\langle F_1, \ldots , F_j \rangle \cap \langle F_{j+1}\rangle)$ and $\redH_*(\langle F_1, \ldots F_j\rangle)$, while  $\redH_*(\langle F_{j+1} \rangle) = 0$. The map in homology  $\redH_*(\langle F_1, \ldots , F_j \rangle \cap \langle F_{j+1}\rangle) \to \redH_*(\langle F_1, \ldots F_j\rangle)$ is trivial for dimension reasons. Then the Mayer--Vietoris sequence yields the isomorphism
\begin{equation*}
\redH_n(\langle F_1, \ldots F_{j+1}\rangle ) \cong \redH_{n-1}(\langle F_1, \ldots F_{j} \rangle \cap \langle F_{j+1} \rangle) \oplus \redH_n(\langle F_1, \ldots F_{j} \rangle).
\end{equation*}
\end{rem}

\begin{rem}
The same style of argument above also proves that the homotopy type of a shellable simplicial set is a wedge of spheres too. The only difference in the argument is that when adding the next simplex $\Delta^n$, its intersection with the complex constructed up to that point is a pure simplicial set of dimension $n-1$ that is also a subcomplex of $\partial \Delta^n$. In this case, though, the homotopy type of a pure $n-1$-dimensional subcomplexes of $\partial \Delta^n$ is either that of a sphere $\partial \Delta^n$ or contractible. But for the case of $q$-complexes, the homotopy type of these intersections can vary much more, being either contractible or a wedge sum of spheres. The number can vary from 0 up to the number of spheres in the $q$-sphere of the facet $F_i$. 
\end{rem}

Since we now know that the homotopy type of the order complex $K(\Deltao)$ for a pure, shellable $q$-complex is a wedge sum of spheres, we can now identify exactly how many spheres are in the wedge sum by calculating the rank of its homology groups. We carry out this calculation in the next section.

%% file: Homology_degree.tex
\section{Homology of some classes of shellable \texorpdfstring{$q$}{q}-complexes}\label{sec:4}


In this section, we make use of the shellability of the order complexes proved in the previous section to determine the singular homology of lexicographically shellable $q$-complexes. It completes the determination of singular homology of $q$-matroid complexes which was partially done in \cite[Theorem 6.10]{GPR21} by a method parallel to the classical one for shellable simplicial complexes. 

\begin{definition}
    Let $K$ be a shellable simplicial complex whose shelling is given by $F_1,\dots,F_t$. For a facet $F_i$ in $K$, the restriction of $F$ is defined as $\R(F_i)=\{x\in F_i\colon F_i\setminus\{x\}\in K_{i-1}\}$, where $K_{i-1}$ is the simplicial complex generated by $F_1,\dots,F_{i-1}$.
\end{definition}
 \begin{pro}\label{PRO:4.4}
 Let $\Delta$ be a lexicographically shellable $q$-complex on $\Fq^n$ of dimension $r$ and $K(\Delta)$ be the corresponding order complex. Let $F_1, \ldots, F_t$ be the shelling associated to $\preceq_q$ on $\Delta$ and $\preceq_l$ be the shelling on the order complex $K(\Delta)$. Then for a maximal chain $\UU = \{U_0 \subset U_1 \cdots \subset U_r=F_j\}$ in $K(\Delta)$, $\R(\UU) = \UU$ if and only if there exists $1 \leq i <j\leq t$ such that $U_{r-1} \subseteq F_i$ and $U_k \neq \min_{\prec_q} \{ A : U_{k-1} \subset A \subset U_{k+1}\}$ for $1 \leq k < r$. 
 \end{pro}

 \begin{proof}
 For a maximal chain $\UU = \{U_0 \subset U_1 \cdots \subset U_r=F_j\}$ in $K(\Delta)$, $\R(\UU) = \UU$ implies that every maximal subchain is contained in a maximal chain of $K(\Delta)$ preceding $\UU$ (w.r.t. the reverse lexicographic order $\preceq_l$). Now we take a maximal subchain $\tilde{\UU}$ of $\UU$ of length $r-1$ which is obtained by removing $U_k$ from $\UU$ for some $ 1\leq k \leq r$. If $k=r$, then $\tilde{\UU}$ is contained in a maximal chain preceding $\UU$ if and only if $U_{r-1} \subseteq F_i$ for some $1 \leq i < j \leq t$. On the other hand, if $k \neq r$, then $U_k \neq \min_{\prec_q} \{ A : U_{k-1} \subset A \subset U_{k+1}\}$, as otherwise, there will be no maximal chain preceding $\UU$ that contains $\tilde{\UU}$. 
 
 Conversely, it is clear that if there exists $1 \leq i <j\leq t$ such that $U_{r-1} \subseteq F_i$ and $U_k \neq \min_{\prec_q} \{ A : U_{k-1} \subset A \subset U_{k+1}\}$ for $1 \leq k < r$, then any maximal subchain of $\UU$ is contained in a previous maximal chain of $K(\Delta)$.
 \end{proof}
 
 \begin{pro}\label{pro:4.4}
 Let $\Delta$ be a lexicographically shellable $q$-complex on $\Fq^n$ of dimension $r$ and $K(\Delta)$ be the corresponding order complex. For any chain $\UU = \{U_0 \subset U_1 \cdots \subset U_t\}$ in $K(\Delta)$, if $U_{t-1}$ contains the minimum nonzero vector of $U_t$, then $$U_s = \min \{A \colon U_{s-1} \subset A \subset U_{s+1}\}$$ for some $1 \leq s < t$.
 \end{pro}
 
 \begin{proof}
 Let $v$ be the minimum nonzero vector in $U_{t}$ and $U_{t-1}$ contains $v$. Thus $v$ is also the minimum nonzero vector of $U_{t-1}.$ Suppose $s$ is the smallest integer with $1 \leq s \leq  t-1$ such that $v \in U_s$. Then $U_s = U_{s-1} \oplus \langle v \rangle $ with $v = \min U_{s+1} \backslash U_{s-1}$. Then Lemma \ref{lem:3} implies that $U_s = \min\{A \colon U_{s-1} < F < U_{s+1}\}$. Indeed, since we proved in Lemma \ref{lem:3} that if for an unrefinable chain $U_{s-1} \subset U_s \subset U_{s+1}$, $U_s = \min_{\prec_q} \{A \colon U_{s-1} < A < U_{s+1}\}$ implies $U_s = U_{s-1} \oplus \<a \>$ where $a = \min U_{s+1} \backslash U_{s-1}$.
 \end{proof}

 \begin{cor}\label{cor:4.5}
 Let $\UU = \{U_0 \subset U_1 \cdots \subset U_r\}$ be a maximal chain in the order complex $K(\Delta)$ of a $q$-complex $\Delta$ of dimension $r$. Then $U_k \neq \min_{\prec_q} \{ A \colon U_{k-1} \subset A \subset U_{k+1}\}$ for all $1 \leq k < r$ if and only if $U_{k}$ does not contain the minimum nonzero vector of $U_{k+1}$ for all $1 \leq k < r$.
 \end{cor}
 \begin{proof}
  Suppose $U_k \neq \min_{\prec_q} \{ A \colon U_{k-1} \subset A \subset U_{k+1}\}$ for all $1 \leq k < r$ and assume that $U_i$ contains the minimum nonzero vector of $U_{i+1}$ for some $1\le i < r$. By Proposition \ref{pro:4.4}, there exists $1 \le s < i+1 \le r $ such that 
  \[
  U_s = \min \{A \colon U_{s-1} \subset A \subset U_{s+1}\},
  \]
  which contradicts the assumption.

  For the converse, assume that $a_{k+1}$ is the minimum nonzero vector of $U_{k+1}$ for all $1 \leq k < r$ and $a_{k+1} \notin U_{k}$ for all $1 \leq k < r$. Note that $a_{k+1} = \min U_{k+1} \backslash U_{k-1}$. Then $U_k \neq \min_{\preceq_q}\{A \colon U_{k-1} \subset A \subset U_{k+1}\}$ for all $1 \leq k \leq r-1$, since $\min_{\preceq_q}\{A \colon U_{k-1} \subset A \subset U_{k+1} = U_{k-1} \oplus \<  \min U_{k+1} \backslash U_{k-1} \>$ as proved in Lemma \ref{lem:3}.
 \end{proof}
 Next we count the number of maximal chains $\UU$ with the property $\R(\UU) = \UU$ in $K(\Delta_q(k,n))$ where $\Delta$ is the uniform $q$-complex of  dimension $k$ on $\Fq^n$. This recovers a result from \cite[Theorem 5.11]{GPR21} on singular homology of $\Delta_q(k,n)$.  
 \begin{pro}\label{pro:4.9}
 Let $\Delta := \Delta_q(k,n)$ be the uniform $q$-complex of dimension $k$ on $\Fq^n$ and $K(\Delta)$ its corresponding order complex. 
 Then 
 \[
 | \{\UU \in K(\Delta) \colon \UU \text{ is maximal and }\R(\UU)=\UU \} |=
 q^{k(k+1)/2} {n-1\brack k}_q.
 \]
 \end{pro}
 
 The idea is to use the characterization of the maximal chains $\UU$ satisfying the property $\R(\UU) = \UU$ as proved in Proposition \ref{PRO:4.4}. We prove the following Lemma that will be useful in proving the Proposition.

 \begin{lem}\label{lem:4.6}
Let $\Delta$ be a lexicographically shellable $q$-complex over $\Fqn$ of dimension $k$. Let $\UU$ be a maximal chain in $K(\Delta)$ ending with a facet $F_j$ of $\Delta$. If $x$ is the minimum (w.r.t. $\prec$) nonzero vector in $\Fq^n$ such that $\langle x \rangle \in \Delta$ and $F_j$ contains $x$, then $\R(\UU) \neq \UU$.
 \end{lem}
 \begin{proof}
     Let $\UU = \{U_0 \subset U_1 \cdots \subset U_k = F_j\}$ and $t$ is the smallest integer $1\leq t \le k$ such that $x=\min_{\preceq U_t\setminus U_{t-1}}$. 
     
     If $t =k$, then $x = \min_{\prec} F_j \setminus U_{k-1}$. Thus $U_{k-1}$ cannot be contained in a facet $F_i$ where $F_i \prec_q F_j$. Indeed,  $U_{k-1} \subseteq F_i$ will imply $F_i \cap F_j = U_{k-1}$ and therefore $x \in F_j \setminus F_i$. Since $x$ is the minimum nonzero element such that $\< x\> \in \Delta$, it is clear that $F_j \prec_q F_i$. 
     
     On the other hand, if $t < k$, then by Lemma \ref{lem:3}, $U_t=\min_{\prec_q}\{A\colon U_{t-1}\subset A \subset U_{t+1}\}$ since $U_t = U_{t-1} \oplus \< x \>$. Therefore, in both the cases, the maximal chain $\UU$ ending with $F_j$ does not satisfy the hypothesis of Proposition \ref{PRO:4.4} and we can conclude that $\R(\UU) \neq \UU $.
 \end{proof}
 \begin{proof}[Proof of the Proposition \ref{pro:4.9}]
 Let $x$ be the minimum nonzero vector in $\Fq^n$ such that $\langle x \rangle \in \Delta$. Following Lemma \ref{lem:4.6}, we know that if a facet $F$ of $\Delta$ contains $x$, then no maximal chain $\UU$ in $K(\Delta)$ ending with $F$ can have the property $\R(\UU) = \UU$. On the other hand, if a facet $F_j$ of $\Delta$ does not contain $x$, then any codimension $1$ subspace of $F_j$ is contained in a facet $F_i$ preceding $F_j$. Indeed, since for any subspace $U_{k-1} \subseteq F_j$ of dimension $k-1$, the subspace $F_i := U_{k-1} \oplus \<x\>$ is a facet of $\Delta=\Delta_q(k,n)$ and $F_i$ must precede $F_j$ since $x \in F_i \setminus F_j$.  Moreover, Corollary \ref{cor:4.5} imply that a maximal chain $\UU = \{U_0 \subset U_1 \cdots \subset U_k = F_j\}$ has the property $\R(\UU) = \UU$ if and only if $\min U_{i+1} \notin U_i $ for all $1 \leq i \leq k-1$. Thus for a fixed facet $F_j$ that does not contain $x$, the number of maximal chains $\UU$ ending with $F_j$ and satisfying the property $\R(\UU) = \UU$ is $({k \brack k-1}_q - {k-1 \brack k-2}_q)\cdots ({2 \brack 1}_q - 1) = q^{k-1}({k-1 \brack k-1}_q) q^{k-2} ({k-2 \brack k-2}_q) \cdots (q+1-1)=q^{k(k-1)/2}$. Now the number of the facets $F_i$ of $\Delta$ not containing the minimum vector $x$ is ${n \brack k}_q - {n-1 \brack k-1}_q = q^{k}{n-1 \brack  k}_q$.
 Thus 
 \begin{align*}
     |\{\UU \in K(\Delta) \colon \UU \text{ is maximal and }\R(\UU)= \UU \} |&= q^{k(k-1)/2} q^k{n-1 \brack k}_q\\
     &= q^{k(k+1)/2} {n-1 \brack k}_q.
 \end{align*}
 \end{proof}
 
 Next we give a method of computing the simplicial homology of the order complex corresponding to \emph{any} lexicographically shellable $q$-complex.
 \begin{thm}
Let $\Delta$ be a lexicographically shellable $q$-complex of dimension $k$ on $\Fq^n$. Suppose $F_1, \ldots, F_t$ are the facets of $\Delta$ ordered according to the shelling $\prec_q$. Let the first $s$ facets contain the minimum nonzero vector $x$ of $\Delta$. Also, for $s+1 \leq j \leq t$, if $x_j$ is the minimum nonzero vector in $F_j$, then we set $r_j = |\{F_i \cap F_j \colon \dim F_i \cap F_j = k-1, \; F_i \prec_q F_j, \text{ and } x_j \notin F_i \cap F_j  \}|$. Then
$$| \{\UU \in K(\Delta) \colon \UU \text{ is maximal and }\R(\UU)= \UU \} |= q^{(k-1)(k-2)/2}  \sum\limits_{i=s+1}^{t} r_i.$$
 \end{thm}
 \begin{proof}
 The proof is essentially same as the proof of Proposition \ref{pro:4.9}. First we note that the facets of $\Delta$ that contains the minimum nonzero vector of $\Delta$ are always in the beginning of the shelling order $\prec_q$. Therefore, the facets of $\Delta$ containing the minimum nonzero vector $x$ of $\Delta$ are precisely the first $s$ facets that contain $x$. Following Lemma \ref{lem:4.6}, it is clear that for the maximal chains $\UU$ ending with any of the first $s$ facets $\R(\UU) \neq \UU$. Now for any $s+1 \le j \le t$, if a maximal chain $\UU= U_0 \subset U_1 \subset \cdots \subset U_{k-1} \subset F_j$ has the property $\R(\UU)=\UU$, then following Proposition \ref{PRO:4.4} and Corollary \ref{cor:4.5} $\UU$ has to satisfy the following:
 \begin{itemize}
\item[(i)] $U_{k-1} \subseteq F_i$ for some $1 \le i < j$,

\item[(ii)] $\min \,U_{i+1} \notin U_i $ for all $1 \le i < k$.
 \end{itemize}
The condition (i) is equivalent to the existence of a facet $F_i$ such that $F_i \prec_q F_j$ with 
$F_i \cap F_j = U_{k-1}$. Also, condition (ii) implies that $ \min F_j \notin U_{k-1}  $, i.e., $x_j \notin F_i \cap F_j$. Thus there are $r_j$ choices for $U_{k-1}$ for a maximal chain $\UU =\{ U_0 \subset \cdots \subset U_{k-1} \subset F_j\}$ for each $s+1 \le j \le t$. For the remaining $U_i$'s with $1 \le i \le k-2$, we use condition (ii) to count the possible number of subchains $ \tilde{\UU} =\{ U_0 \subset U_1 \subset \cdots \subset U_{k-2}\}$ that can be extended to maximal chains $\UU = \{\tilde{\UU} \subset U_{k-1} \subset F_j\}$ with the desired property $\R(\UU) =\UU$. The number of such subchains $\tilde{\UU}$ is $({k-1 \brack k-2}_q - {k-2 \brack k-3}_q)\cdots ({2 \brack 1}_q - 1) = q^{(k-1)(k-2)/2}$. Thus combining with the choices of $U_{k-1}$ we can conclude that the number of the maximal chains satisfying $\R(\UU) =
\UU$ is $q^{(k-1)(k-2)/2}  \sum\limits_{i=s+1}^{t} r_i.$
 \end{proof}
 \begin{cor}
 Let $\Delta$ be any lexicographically shellable $q$-complex of dimension $k$ on $\Fq^n$. Suppose $F_1, \ldots, F_t$ are the facets of $\Delta$ ordered according to the shelling $\prec_q$. Let the first $s$ facets contain the minimum nonzero vector $x$ of $\Delta$. Also, for $s+1 \leq j \leq t$, if $x_j$ be the minimum nonzero vector in $F_j$, then we set $r_j = |\{F_i \cap F_j \colon \dim F_i \cap F_j = k-1, \; F_i \prec_q F_j, \text{ and } x_j \notin F_i \cap F_j  \}|$. Then
 \[
\widetilde{H_p}(\ring{\Delta}) = \begin{cases}
\ZZ^{q^{(k-1)(k-2)/2}  \sum\limits_{i=s+1}^{t} r_i} & \text{if } p=k-1,\\
0 & \text{ otherwise}.
\end{cases}
\]  
 \end{cor}
\begin{proof}
This follows from \cite[Theorem 7.7.2]{BJO92}.
\end{proof}

\begin{rem}
Note that the way we calculated the homology of these spaces really just uses the axioms for homology \cite[\S2.3]{Hatcher}. This shows that it doesn't matter which version of the spaces $\Deltao$ you work with. Whether you work with $T(\Deltao)$, $N(\Deltao)$, or $B(\Deltao)$, what is important is just the way the space is glued together from simple pieces. A homology theory in any of these categories needs to satisfy the basic axioms of a homology theory, and these are all we use in the calculation. What is both interesting and fortuitous is that the argument for computing the homology of these spaces can be reworked to also identify their homotopy type.  There do exist spaces that have the same homology as a sphere, but are not in fact weak equivalent to spheres. 
\end{rem}


We end this section by determining the singular homology for the $q$-matroid complex in \cite[Example 6.13]{GPR21}. This $q$-complex was shown to not satisfy the hypothesis of \cite[Theorem 6.10]{GPR21}, but we show that it is possible to `rearrange' its facets to apply the technique in \cite{GPR21} for determining singular homology.  
 \begin{exa}\label{exa:main}
Consider the field extension $\mathbb{F}_{2^4}/\mathbb{F}_{2}$ of degree $4$, and let $a$ be a root in $\mathbb{F}_{2^4}$ of the irreducible polynomial $X^4 + X+ 1$ in $\mathbb{F}_{2}[X]$ so that $\mathbb{F}_{2^4} = \mathbb{F}_{2}(a)$. 
Let $C$ be the rank metric code of length 4 over the extension $\mathbb{F}_{2^4}$ of $\mathbb{F}_{2}$ such that a generator matrix of $C$ is given by 
\[ 
G:= 
\begin{pmatrix}
a^{2} + a + 1 \ & a^{2} & a^{3} + a + 1 \ & a^{3} + a^{2} + a + 1 \\
a^{2} + a + 1 & a^{3} + 1 \ & a & a + 1 \\
a^{2} + 1 & 1 & a^{2} + 1 & a^{3} + 1
\end{pmatrix}.
\]
Let $\Delta_C$ be the $q$-matroid complex on $\mathbb{F}_{2}^4$ associated to $C$ as in Example 2.9.
Then 
$\dim \Delta_C = \mathop{\rm rank}(G)=3$. There are ${{4}\brack{3}}_2 = 15$ subspaces of $\mathbb{F}_{2}^4$ of dimension $3$ and it turns out that 14 among these are in $ \Delta_C$. In the shelling order of Definition 3.5, these 14 facets of $ \Delta_C$, say $F_1, \dots , F_{14}$, can be explicitly listed as follows. 
$$
\begin{array}{l}
\langle \mathbf{e}_2, \, \mathbf{e}_3, \, \mathbf{e}_4 \rangle,  \ \,
\langle \mathbf{e}_1 + \mathbf{e}_2, \, \mathbf{e}_3, \,  \mathbf{e}_4 \rangle, \ \, 
\langle \mathbf{e}_1, \,  \mathbf{e}_2, \, \mathbf{e}_4 \rangle, \ \, 
\langle \mathbf{e}_1 + \mathbf{e}_3, \,  \mathbf{e}_2, \,  \mathbf{e}_4 \rangle, \ \, 
\langle \mathbf{e}_1,  \, \mathbf{e}_2 + \mathbf{e}_3, \, \mathbf{e}_4 \rangle, \\
\langle \mathbf{e}_1 + \mathbf{e}_3,  \, \mathbf{e}_2 + \mathbf{e}_3, \, \mathbf{e}_4 \rangle, \ \, 
\langle \mathbf{e}_1, \, \mathbf{e}_2, \, \mathbf{e}_3 \rangle,  \ \, 
\langle \mathbf{e}_1 + \mathbf{e}_4, \, \mathbf{e}_2, \,  \mathbf{e}_3 \rangle, \ \, 
\langle \mathbf{e}_1, \,  \mathbf{e}_2 + \mathbf{e}_4 , \, \mathbf{e}_3 \rangle, \\
\langle \mathbf{e}_1 + \mathbf{e}_4,  \,  \mathbf{e}_2 +  \mathbf{e}_4, \,  \mathbf{e}_3 \rangle, \ \, 
\langle \mathbf{e}_1 ,  \, \mathbf{e}_2 , \, \mathbf{e}_3 + \mathbf{e}_4 \rangle, \  \, 
\langle \mathbf{e}_1 + \mathbf{e}_4, \, \mathbf{e}_2, \, \mathbf{e}_3 + \mathbf{e}_4 \rangle, \\
\langle \mathbf{e}_1 , \,  \mathbf{e}_2 +  \mathbf{e}_3, \,    \mathbf{e}_3  + \mathbf{e}_4 \rangle, \ \, 
\langle \mathbf{e}_1  + \mathbf{e}_4,  \, \mathbf{e}_2 +  \mathbf{e}_4 , \,    \mathbf{e}_3  + \mathbf{e}_4 \rangle, 
\end{array}
$$
where for $1\le i \le 4$, by $\mathbf{e}_i$ we have denoted the element of $\mathbb{F}_{2}^4$ with $1$ in the $i$th position and $0$ elsewhere. We can take 
a generator matrix of $F_j$ to be 
the $3\times 4$ matrix $Y_j$, which has as its rows the elements of the given ordered basis of $F_j$, and it can be checked that the rank of  the $3 \times 3$ matrix $GY_j^T$ is indeed $3$ for each $j=1, \dots , 14$. Incidentally, 
 the only $3$-dimensional subspace of $\mathbb{F}_{2}^4$ missing in the above list is $F:=\langle \mathbf{e}_1, \, \mathbf{e}_3, \, \mathbf{e}_4 \rangle$ and 
 its generator matrix $Y$ has the property that 
 rank$(GY^T)=2$; indeed, 
$$
Y = \begin{pmatrix}
1 & 0 & 0 & 0 \\
0 & 0 & 1 & 0 \\
0 & 0 & 0 & 1
\end{pmatrix}
\quad
\text{and} \quad
(GY^T)
\begin{pmatrix}
1 \\  a^{3} + a^{2} + a + 1 \\  a^{2} + a 
\end{pmatrix}
= 
\begin{pmatrix}
0 \\ 0 \\ 0
\end{pmatrix}.
$$  
 \end{exa}
 We consider a new ordered basis $\beta := (\mathbf{e}^{\prime}_1, \mathbf{e}^{\prime}_2, \mathbf{e}_3, \mathbf{e}_4)$ of $\F_2^4$, where $\mathbf{e}^{\prime}_1 = (0, 1, 0 ,0)$ and $\mathbf{e}^{\prime}_2 = (1, 0, 0 ,0)$. With respect to the new basis, the space $F$, the only $3$-dimensional space missing in the list of facets of $\Delta$, is $\langle \mathbf{e}^{\prime}_2, \mathbf{e}_3, \mathbf{e}_4\rangle$. 
 

 \begin{pro}
 Let $\Delta_{C}$ be the $q$-complex of dimension $3$ as in Example \ref{exa:main}.
 Then 
\[
\widetilde{H_p}(\ring{\Delta_{C}}) = \begin{cases}
\ZZ\strut^{7 q^{3}}& \text{if } p=2,\\
0 & \text{otherwise.} 
\end{cases}
\] 
 \end{pro}
 
 \begin{proof}
 Let $J' := \{j \in \{1,\ldots, 14\}~|~ \mathbf{e}_2 \in F_j\}$ and $J := \{1, \ldots, 14\} \setminus J'$. Since all the $3$-dimensional spaces of $\F_2^4$ except $\langle \mathbf{e}_1, \mathbf{e}_3, \mathbf{e}_4\rangle$ are facets of $\Delta_{C}$, we have $|J'| = {{3}\brack{2}}_2 = 7$ and $|J|= 7$. Consider the subcomplex $\Delta^{\prime} := \langle F_j ~ | ~ j \in J'\rangle$. Since $\mathbf{e}_2$ is contained in all the facets of $\Delta^{\prime}$, the punctured $q$-complex $\ring{\Delta}^{\prime}$ is contractible (by \cite[Lemma 5.3]{GPR21}). 
 
 For any $j \in J$, let $U$ be a $2$-dimensional subspace of $F_j$. Since $\mathbf{e}_2 \notin F_j$, $\dim_{\F_2} U \oplus \langle \mathbf{e}_2 \rangle = 3$. Therefore, $U \oplus \langle \mathbf{e}_2 \rangle = F_s$ for some $s \in J'$ and hence $U \in \Delta^{\prime}$. This implies that $\ring{\Delta}^{\prime} \cup \ring{\Sigma}(F_{j}) = \ring{S}_q^{2}$ for any $j \in J$. Recall that 
 \[
\widetilde{H_p}(\ring{S}_q^{2}) = \begin{cases} 
\ZZ\strut^{q^{3}}& \text{if } p=1,\\
0 & \text{otherwise.} 
\end{cases}
\] 
Now using the suitable Mayer-Vietoris sequences and proceeding as in the proof of \cite[Theorem 5.7]{GPR21}, we obtain the desired result about the reduced homology groups of $\Delta_{C}$.
 \end{proof}
 
 \begin{rem}
 It is indeed possible to `rearrange' the $3$-dimensional subspaces of $\F_2^4$ such that the only nonface appears at the end. Corresponding to the total order we considered on the Grassmannian $\mathbb{G}_3(\F_2^4)$, the space $\langle \mathbf{e}_1, \mathbf{e}_2, \mathbf{e}_3 \rangle$ is the one that appears at the end. In our example $\langle \mathbf{e}_1, \mathbf{e}_3, \mathbf{e}_4 \rangle$ is the nonface. Considering an isomorphism of $\F_2^4$ which sends $\langle \mathbf{e}_1, \mathbf{e}_3, \mathbf{e}_4 \rangle$ to $\langle \mathbf{e}_1, \mathbf{e}_2, \mathbf{e}_3 \rangle$, we rearrange the $3$-dimensional subspaces so that the nonface is now moved to the end for the new $q$-complex which is isomorphic to the one we considered in the example. Note that the posets and the order topology corresponding to two isomorphic $q$-complexes are also isomorphic.  
 \end{rem}

%% file: Conclusion_and_open_question.tex
\section{Conclusion and open question}\label{sec:5}
In this paper, we proved that shellable $q$-complexes have homotopy type of wedge sum of $q$-spheres. Moreover, we determined explicitly the singular homology groups of a class of shellable $q$-complexes that we call \emph{lexicographically shellable} $q$-complexes. This completes the study in \cite{GPR21}, whereas we leave open the question of shellability of order complexes of \emph{any} shellable $q$-complexes. 

Moreover, looking back at the motivation behind the work in \cite{GPR21} from a coding theory perspective, an unresolved question remains: What is the connection between the generalized rank weights of a vector rank metric code and the singular homology of the associated $q$-matroid complex? More details on this inquiry can be found in \cite[Introduction]{GPR21}. We must mention that a recent study \cite{JPV} has indeed established a link between the generalized rank weights of a rank metric code and a set of invariants, called as 'Virtual Betti numbers', of the associated $q$-matroids. These virtual Betti numbers are in fact the Betti numbers of the Stanley-Reisner rings of the classical matroids corresponding to the $q$-matroids \cite{JPV}. Thus, the question of relating the singular homology computed in this paper to the generalized rank weights is reduced to identifying the missing connection: How does the singular homology of a $q$-matroid complex relate to the $\mathbb{N}$-graded Betti numbers of the Stanley-Reisner rings of its associated classical matroid?